\documentclass{jaums}

\usepackage{amsmath}
\usepackage{amssymb}
\usepackage{MnSymbol}
\usepackage{dsfont}
\usepackage{epsfig}  		
\usepackage{xypic}
\usepackage{bbm}
\usepackage{hyperref}
\usepackage{eucal}
\usepackage{mathrsfs}
\usepackage{xcolor}
\usepackage{lineno}

\theoremstyle{thmit} 

\newtheorem{mthm}{Theorem}

\newtheorem{thm}{Theorem}[section]
\newtheorem{lemma}[thm]{Lemma}

\theoremstyle{thmrm} 
\newtheorem{example}[thm]{Example}

\newtheorem*{remark*}{Remark}
\newtheorem*{oldproof}{Proof}
\renewenvironment{proof}[1][{}]{\begin{oldproof}[#1]}{\qed\end{oldproof}}

\let\rk\relax\DeclareMathOperator{\rk}{\mathrm{rk}}

\DeclareMathOperator{\Span}{\mathrm{span}}

\newcommand{\CC}{\mathds{C}}
\newcommand{\RR}{\mathds{R}}

\newcommand{\ZZ}{\mathds{Z}}

\newcommand{\ad}{\mathrm{ad}}
\newcommand{\Ad}{\mathrm{Ad}}

\renewcommand{\phi}{\varphi}
\renewcommand{\epsilon}{\varepsilon}

\newcommand{\group}{\mathrm}

\newcommand{\GL}{\group{GL}}

\newcommand{\Hol}{\mathrm{Hol}}

\newcommand{\Gtwo}{\group{G}_{2(2)}}
\newcommand{\G}{\group{G}}

\renewcommand{\rho}{\varrho}

\newcommand{\frg}{\mathfrak{g}}
\newcommand{\frh}{\mathfrak{h}}

\newcommand{\frb}{\mathfrak{b}}

\newcommand{\frn}{\mathfrak{n}}
\newcommand{\frnI}{\mathfrak{n}_{\text{\sc i}}}
\newcommand{\frnII}{\mathfrak{n}_{\text{\sc ii}}}
\newcommand{\frnIII}{\mathfrak{n}_{\text{\sc iii}}}

\renewcommand{\frm}{\mathfrak{m}}
\newcommand{\frw}{\mathfrak{w}}

\newcommand{\frj}{\mathfrak{j}}

\newcommand{\zen}{\mathfrak{z}}

\renewcommand{\sl}{\mathfrak{sl}}

\newcommand{\sso}{\mathfrak{so}}

\newcommand{\der}{\mathfrak{der}}

\newcommand{\zsp}{\mathbf{0}}

\newcommand{\met}{\langle\cdot,\cdot\rangle}

\newcommand{\tensor}{\mathsf}

\newcommand{\g}{\tensor{g}}
\newcommand{\R}{\tensor{R}}

\DeclareMathOperator{\im}{\mathrm{im}}

\newcommand{\Tan}{\mathrm{T}}

\numberwithin{equation}{section}

\title[On compact homogeneous $\Gtwo$-manifolds]{On compact homogeneous $\boldsymbol{\Gtwo}$-manifolds}
\author{Wolfgang Globke}
\address{Wolfgang Globke\\Faculty of Mathematics\\Oskar-Morgenstern-Platz 1\\Universit\"at Wien\\1090 Vienna\\Austria\\{\sf wolfgang.globke@univie.ac.at}}

\thanks{Wolfgang Globke is supported by the Austrian Science Fund FWF grant I 3248.
}

\begin{document}


\maketitle

\begin{abstract}
We prove that among all compact homogeneous spaces for an effective
transitive action of a Lie group whose Levi subgroup has no compact
simple factors,
the seven-dimensional flat torus is the only one that admits an
invariant torsion-free $\Gtwo$-structure.
\end{abstract}




\section{Introduction and main results}
\label{sec:intro}

Let $G$ be a connected Lie group whose Levi subgroup has no compact
simple factors.
Let $M$ be a seven-dimensional connected compact homogeneous space for
$G$, by which we mean $M=G/H$ for some closed uniform subgroup $H$ of
$G$, and $G$ acts effectively on $M$.
Assume now that there exists a $G$-invariant torsion-free
$\Gtwo$-structure $(\g_M,\phi_M)$ on $M$,
where $\Gtwo$ denotes the split real form of the complex exceptional
Lie group $\G_2^{\CC}$.
This means $\g_M$ is a pseudo-Riemannian metric of signature $(4,3)$,
and $\phi_M$ is a certain parallel three-form whose stabilizer on
each tangent space is the group $\Gtwo$.

In this article we will show the following:

\begin{mthm}\label{mthm:flat}
Let $M$ be a compact seven-dimensional homogeneous space for a Lie
group $G$ whose Levi subgroup does not have compact simple factors.
If $M$ has a $G$-invariant torsion-free $\Gtwo$-structure, then
$G=M=\RR^7_3/\ZZ^7$ is a flat torus.
In particular, the holonomy group of $M$ is trivial.
\end{mthm}

The proof of this theorem makes use of results by Baues, Globke and
Zeghib \cite{BGZ} to reduce it to the case of a Lie group with
bi-invariant metric. Then it follows from the next theorem:

\begin{mthm}\label{mthm:biinvariant}
Let $G$ be a seven-dimensional Lie group and $\Gamma$ a uniform lattice
in $G$.
If $M=G/\,\Gamma$ has a $G$-invariant torsion-free $\Gtwo$-structure whose associated metric pulls back to a bi-invariant metric on $G$, then
$G=M=\RR^7_3/\ZZ^7$ is a flat torus.
In particular, the holonomy group of $M$ is trivial.
\end{mthm}

In Section \ref{sec:nilpotent} we use a result by Globke and
Nikolayevsky \cite{GN} to reduce the problem to nilpotent Lie groups.
The nilpotence of $G$ then implies that $\phi_M$ induces an
$\Ad(G)$-invariant three-form on the Lie algebra of $G$.
We then investigate in Section \ref{sec:liealgebra} the obstructions
for nilpotent Lie algebras with invariant scalar products to be
contained in $\frg_{2(2)}$.
Combining these results, we prove Theorems \ref{mthm:flat} and
\ref{mthm:biinvariant} in Section \ref{sec:proofs}.

Contrasting our theorems here, there are many examples of homogeneous
spaces for compact Lie groups that admit non-trivial invariant $\G_2$- or
$\Gtwo$-structures.
L\^e and Munir \cite{LM} established a classification of these
spaces.
In the Riemannian $\G_2$-case, this is a complete classification
of compact homogeneous $\G_2$-manifolds.
See the references in \cite{LM} for further related results.

Other existence results in a setting similar to ours were obtained,
for example, by Fino and Luj\'an \cite{FL}, who determine $\Gtwo$-structures on
compact nilmanifolds $\Gamma\,\backslash G$, or equi\-valently,
left-invariant $\Gtwo$-structures on seven-dimensional nilpotent Lie
groups $G$.
Note that, despite the authors calling these structures ``invariant'',
the metrics on these quotients are not $G$-invariant.
In fact, the additional assumption made in \cite{FL}
that the metric is definite on the center of $G$
prohibits the pulled-back metric on $G$ to be bi-invariant,
which is necessary for the metric on $\Gamma\,\backslash G$ to be
$G$-invariant (unless $G$ is abelian to begin with).

Another related result is the classification of indecomposable
indefinite symmetric spaces with $\Gtwo$-structures by Kath
\cite{kath2}.
All of these turn out to be non-compact quotients of nilpotent Lie
groups with a bi-invariant metric, and their holonomy groups
are three-dimensional and abelian.

\subsection*{Acknowledgement}
I would like to thank Thomas Leistner for helpful discussions
of holonomy groups, and also the anonymous referee for many helpful
suggestions that improved the readability of the paper.

\subsection*{Notations and conventions}
A Lie group or Lie algebra is called \emph{$k$-step nilpotent} if the
$k$th term in its descending central series is trivial, but the
$(k-1)$st term is not.

For a Lie algebra $\frg$ with subalgebra $\frh$ we let $\ad(\frg)$
denote the adjoint representation of $\frg$, and when clarity requires
it we write $\ad_\frg(\frh)$ for the adjoint action of $\frh$ on $\frg$
to distinguish it from the adjoint representation of $\frh$ on itself.
A similar notation is used for Lie groups.

A Lie algebra is called \emph{metric} if it has an invariant scalar
product. The vector space $\RR^n$, equipped with a scalar product of
signature $(n-s,s)$, is denoted by $\RR^n_s$.


\section{Reduction to nilpotent groups}
\label{sec:nilpotent}

In this section we will show that for our purposes it is sufficient
to consider nilpotent Lie groups.

\begin{lemma}\label{lem:Levi_noncompact}
Let $M$ be a compact homogeneous space for a Lie
group $G$ whose Levi subgroup does not have compact simple factors.
Suppose $G$ acts effectively on $M$.
If $M$ has a $G$-invariant pseudo-Riemannian metric,
then $M=G/\,\Gamma$ for some uniform lattice $\Gamma$ in $G$, and
in particular $\dim G=\dim M$.
Moreover, the pseudo-Riemannian metric
on $M$ is induced by a bi-invariant metric on $G$.
\end{lemma}
\begin{proof}
Write $M=G/H$ for a closed uniform subgroup $H$ of $G$.
Via pullback to $G$, $\g_M$ induces a symmetric bilinear form on the
Lie algebra of $G$ whose kernel is the Lie algebra of $H$.
Baues, Globke and Zeghib \cite[Theorem A]{BGZ} showed that
this bilinear form is $\Ad(G)$-invariant, so that its kernel is an ideal.
By the effectivity of the $G$-action, this ideal, and hence the
identity component of $H$, must be trivial.
This means $H$ is a uniform lattice in $G$, so that
$\dim G=\dim M$. By the invariance of the bilinear form
 on the Lie algebra, the metric on $M$ is induced by a bi-invariant
 metric on $G$.
\end{proof}

The existence of a torsion-free $\Gtwo$-structure $(\g_M,\phi_M)$ on a manifold $M$
means that the three-form $\phi_M$ is parallel and is thus preserved
by the holonomy group at every point $p\in M$.
The stabilizer of $(\phi_M)_p$ in $\GL(\Tan_p M)$ is $\Gtwo$,
which means $\Hol(\g_M)\subseteq\Gtwo$.
Moreover, the metric $\g_M$ of the $\Gtwo$-structure has vanishing
Ricci curvature (Bonan \cite{bonan}).

\begin{lemma}\label{lem:dimG7}
Let $G$ be a seven-dimensional connected Lie group with a bi-invariant metric $\g_G$,
and $\Gamma$ a uniform lattice in $G$.
Let $M=G/\,\Gamma$ and let $\g_M$ be the metric induced on $M$ by
$\g_G$.
If $\g_M$ is Ricci-flat, then:
\begin{enumerate}
\item
$G$ is nilpotent.
\item
The connected holonomy group of $\g_M$ is
$\Hol(\g_M)^\circ\cong\Ad_\frg([G,G])$.
\end{enumerate}
In particular, this holds if $\g_M$ is the metric of a torsion-free
$\Gtwo$-structure $(\g_M,\phi_M)$ on $M$, and in this case $\Ad_\frg(G)$
is a subgroup of $\Gtwo$.
\end{lemma}
\begin{proof}
The bi-invariance of $\g_G$ implies that the Ricci tensor,
restricted to the left-invariant vector fields on $G$, is proportional
to the Killing form of $G$.
Since $\g_M$, and hence $\g_G$, is Ricci-flat, the Killing form of $G$
is zero, which means $G$ is solvable.
Moreover, the Ricci-flatness implies that $M$ is an Einstein manifold.
It was shown by Globke and Nikolayevsky \cite[proof of Theorem 1.3]{GN}
that a solvable Lie group acting transitively on a compact pseudo-Riemannian
Einstein mani\-fold of dimension less or equal to seven is nilpotent.
So $G$ is nilpotent.

As a Lie group with bi-invariant metric, the curvature tensor on $G$ is
given by $\R(X,Y)Z=\frac{1}{4}[[X,Y],Z]$ for any left-invariant
vector fields $X,Y,Z$ on $G$.
The Ambrose-Singer Theorem (cf.~Besse \cite[Theorem 10.58]{besse})
now implies that the connected holonomy group of $G$ is
$\Hol(\g_G)^\circ\cong\Ad_\frg([G,G])$.
Since $G$ is a covering space of $M$,
$\Hol(\g_G)^\circ\cong\Hol(\g_M)^\circ$ (cf.~Besse \cite[10.16]{besse}).

Suppose $(\g_M,\phi_M)$ is a torsion-free $\Gtwo$-structure on $M$.
Then $\g_M$ and $\g_G$ are Ricci-flat, so all of the above
applies.
The three-form $\phi_M$ pulls back to a left-invariant three-form
$\phi_G$ on $G$ that is right-invariant under $\Gamma$. This means
the induced three-form $\phi_{\frg}$ on $\frg$ is
$\Ad_\frg(\Gamma)$-invariant. But the lattice $\Ad_{\frg}(\Gamma)$ in
the nilpotent Lie group $\Ad_{\frg}(G)$ is Zariski-dense and
$\phi_{\frg}$ is a polynomial expression, hence $\phi_{\frg}$ is
$\Ad_\frg(G)$-invariant.
Since the real algebraic group $\Gtwo$ is the stabilizer of $\phi_{\frg}$,
this means $\Ad_\frg(G)$ is a subgroup of $\Gtwo$.
\end{proof}


\section{Nilpotent metric Lie algebras}
\label{sec:liealgebra}

%
\subsection{Nilpotent metric Lie algebras in low dimensions}

Let $\frn$ be a nilpotent Lie algebra with invariant scalar product
$\met$ of signature $(p,q)$.
Let $\zen(\frn)$ denote the center of $\frn$.
Set
\begin{equation}
\frj=\zen(\frn)\cap[\frn,\frn].
\label{eq:j}
\end{equation}
This is a totally isotropic ideal in $\frn$.
Also, $\frj^\perp$ is an ideal in $\frn$ that contains
$[\frn,\frn]$ and $\zen(\frn)$.
We have vector space decompositions
\begin{equation}
\frn = \frj^* \oplus \frj^\perp,
\quad
\frj^\perp = \frw\oplus\frj,
\label{eq:decomp}
\end{equation}
where $\frj^*$ is totally isotropic, dually paired with $\frj$ via
$\met$ and orthogonal to $\frw$,
and $\met$ has signature $(p-\dim\frj,q-\dim\frj)$ on $\frw$.

The decomposition \eqref{eq:decomp} is helpful in understanding
the algebraic structure of Lie algebras with invariant scalar products.
The following is a simple yet very useful property of the ideal $\frj$.

\begin{lemma}\label{lem:abelian}
$\frn$ is abelian if and only if $\frj=\zsp$.
\end{lemma}
\begin{proof}
If $\frn$ is abelian, then clearly $\frj=\zsp$.
If $\frn$ is not abelian, its descending central series eventually
intersects $\zen(\frn)$ non-trivially, which means $\frj\neq\zsp$.
\end{proof}

\begin{lemma}\label{lem:dim4}
If $\dim\frn\leq 4$, then $\frn$ is abelian.
\end{lemma}
\begin{proof}
The claim is trivial for $\dim\frn\leq 2$.
Suppose first that $\dim\frn=3$.
By \eqref{eq:decomp} and Lemma \ref{lem:abelian}, if $\frn$ was not
abelian, then $\dim\frj=\dim\frj^*=\dim\frw=1$. With the invariance of
$\met$, these dimensions imply $\frj=[\frj^*,\frw]\perp\frj^*$,
contradicting the definition of $\frj^*$.

Now assume $\dim\frn=4$ and $\frn$ is not abelian.
If $\dim\frj=2$, then $\frn=\frj^*\oplus\frj$
and invariance of $\met$ implies
$[\frn,\frn]=[\frj^*,\frj^*]\perp\frj^*$, but also $[\frj^*,\frj^*]\subset\frj$, a contradiction.
This leaves the case $\dim\frj=1$. Here, invariance and $\dim\frj^*=1$
imply $[\frj^*,\frw]\subseteq\frw$.
By nilpotence of $\frn$, $[\frw,\frw]\subseteq\frj$, and since $\frj^*$
is dually paired with $\frj$, it follows that $\frj^*$ acts
non-trivially on $\frw$. But invariance then requires
$[\frj^*,\frw]=\frw$, contradicting the nilpotence of $\frn$.
\end{proof}

With regard to Lemma \ref{lem:dimG7}, we are interested in the question
when for non-abelian $\frn$ of dimension seven,
$\ad_{\frn}(\frn)$ is contained
in $\frg_{2(2)}$.
So let us now further assume that $\dim\frn=7$ and that $\frn$
is $k$-step nilpotent.
Since it is possible that $\frn$ is \emph{decomposable}, meaning it
can be further decomposed into an orthogonal direct sum
$\frn=\frn_1\times\frn_0$ of ideals, we are also interested in
non-abelian nilpotent Lie algebras $\frn_1$ of dimension $\dim\frn_1\leq 7$
with an invariant scalar product.

As is evident from Kath's classification \cite[Theorem 4.7]{kath1},
the Examples \ref{ex:kath6a} to \ref{ex:kath8} below are the only non-abelian
nilpotent Lie algebras $\frn_1$ of dimension less than eight with an
invariant scalar product that
cannot be further decomposed into orthogonal direct sums of ideals.
We use the following notation: If $d=\dim\frj$, let
$a_1,\ldots,a_d\in\frj^*$, $w_1,\ldots,w_{n-2d}\in\frw$ and
$z_1,\ldots,z_d\in\frj$ be bases of the respective subspaces in
\eqref{eq:decomp}, such that the $a_i$ and $z_j$ are dual bases to
each other with respect to $\met$.
The number $\epsilon=\pm 1$ depends on whether $p\geq q$ (then
$\epsilon=1$) or $p<q$ (then $\epsilon=-1$).

\begin{example}\label{ex:kath6a}
Let $\dim\frn_1=7$, $\dim\frj=2$, with bracket relations
\begin{gather*}
[a_1,a_2]=w_1,
\quad
[a_1,w_1]=w_2,
\quad
[a_1,w_2]=-\epsilon w_3,
\quad
[a_1,w_3]=-z_2, \\
[a_2,w_1]=0,
\quad
[a_2,w_2]=0,
\quad
[a_2,w_3]=z_1,
\\
[w_1,w_2]=\epsilon z_1,
\quad
[w_1,w_3]=0,
\quad
[w_2,w_3]=0,
\end{gather*}
where the scalar product on $\frw$ is given by
\[
\langle w_1,w_1\rangle=
\langle w_1,w_2\rangle=
\langle w_2,w_3\rangle=
\langle w_3,w_3\rangle=0,\quad
\langle w_1,w_3\rangle=1,\quad
\langle w_2,w_2\rangle=\epsilon.
\]
We denote this Lie algebra by $\frnI$.
Here, $\frnI$ is three-step nilpotent, $\dim\frw=3$, and the algebra
$\ad_{\frn}(\frw)$ and three-dimensional and abelian.
\end{example}

\begin{example}\label{ex:kath7a}
Let $\dim\frn_1=6$, $\dim\frj=3$, with non-zero bracket relations
\[
[a_1,a_2]=z_3,
\quad
[a_2,a_3]=z_1,
\quad
[a_3,a_1]=z_2.
\]
We denote this Lie algebra by $\frnII$.
Here, $\frnII$ is two-step nilpotent and $\frw=\zsp$.
The invariant scalar product has split signature $(3,3)$.
\end{example}

\begin{example}\label{ex:kath8}
Let $\dim\frn_1=5$, $\dim\frj=2$, with non-zero bracket relations
\[
[a_1,a_2]=w,
\quad
[a_1,w]=-\epsilon z_2,
\quad
[a_2,w]=\epsilon z_1
\]
We denote this Lie algebra by $\frnIII$.
Here, $\frnIII$ is three-step nilpotent, $\dim\frw=1$.
The signature of $\frnIII$ is $(3,2)$ or $(2,3)$, depending on
$\epsilon$.
\end{example}

\subsection[Nilpotent subalgebras of $\frg_{2(2)}$]{Nilpotent subalgebras of $\boldsymbol{\frg_{2(2)}}$}

The Lie algebra $\sso(\frn,\met)$ of skew-symmetric linear maps for
$\met$ can be identified with $\sso_{4,3}$.
Since $\met$ is invariant we have
$\ad(\frn)\subset\der(\frn)\cap\sso_{4,3}$.

The exceptional simple Lie algebra $\frg_{2(2)}$ is also a subalgebra
of $\sso_{4,3}$. It is the stabilizer subalgebra in $\sso_{4,3}$ of a
certain three-form $\phi$ on $\frn$.
In a suitable Witt basis of $\frn$ with respect to $\met$, we identify
$\frg_{2(2)}$ with the following Lie subalgebra of $\sso_{4,3}$:
\[
\frg_{2(2)}=
\Biggl\{
\left(\begin{matrix}
{u_7} & {u_9} & {u_{10}} & {u_{12}} & {u_{13}} & {u_{14}} & 0\\
{u_1} & {u_8} & {u_{11}} & \frac{u_{10}}{2} & -\frac{u_{12}}{4} & 0 & -{u_{14}}\\
{u_2} & {u_3} & {u_7}-{u_8} & -\frac{u_{9}}{2} & 0 & \frac{u_{12}}{4} & -{u_{13}}\\
{u_4} & 4{u_2} & -4{u_1} & 0 & \frac{u_{9}}{2} & -\frac{u_{10}}{2} & -{u_{12}}\\
{u_5} & -2{u_4} & 0 & 4{u_1} & {u_8}-{u_7} & -{u_{11}} & -{u_9}\\
{u_6} & 0 & 2{u_4} & -4{u_2} & -{u_3} & -{u_8} & -{u_9}\\
0 & -{u_6} & -{u_5} & -{u_4} & -{u_2} & -{u_1} & -{u_7}
\end{matrix}\right)
\ \Biggl|\
u_1,\ldots,u_{14}\in\RR
\Biggr\}
\]
(this matrix representation is borrowed from Leistner, Nurowski and
Sagerschnig \cite{LNS}, with some adjustments to the parameter
labelling).

Any strictly triangular subalgebra of $\frg_{2(2)}$ is
contained in a maximal triangular subalgebra of $\frg_{2(2)}$.
All maximal triangular subalgebras of $\frg_{2(2)}$ are conjugate
to each other by $\Gtwo$, so we may assume the strictly triangular
subalgebra $\ad(\frn)$ is conjugate to a subalgebra of
\begin{equation}
\frm=
\Biggl\{
\left(\begin{matrix}
0 & 0 & 0 & 0 & 0 & 0 & 0\\
{u_1} & 0 & 0 & 0 & 0 & 0 & 0\\
{u_2} & {u_3} & 0 & 0 & 0 & 0 & 0\\
{u_4} & 4{u_2} & -4{u_1} & 0 & 0 & 0 & 0\\
{u_5} & -2{u_4} & 0 & 4{u_1} & 0 & 0 & 0\\
{u_6} & 0 & 2{u_4} & -4{u_2} & -{u_3} & 0 & 0\\
0 & -{u_6} & -{u_5} & -{u_4} & -{u_2} & -{u_1} & 0
\end{matrix}\right)
\
\Biggl|\
u_1,\ldots,u_6\in\RR
\Biggr\}.
\label{eq:triangular}
\end{equation}
In fact, this subalgebra $\frm$ is a maximal strictly triangular subalgebra
of a Borel subalgebra of $\frg_{2(2)}$, and hence a conjugate of $\ad(\frn)$
is contained in $\frm$.

\begin{lemma}\label{lem:rank2}
Let $A$ be the matrix in \eqref{eq:triangular}.
Suppose $\rk A\leq2$. Then $u_1=0$, and in addition one of the
following holds,
\begin{enumerate}
\item[\rm(a)]
either $u_2=u_4=0$,
\item[\rm(b)]
or $u_2\neq 0, u_4\neq0$ and $u_3=\frac{4u_2^2}{u_4}$, $u_5=-\frac{u_4^2}{2 u_2}$.
\end{enumerate}
\end{lemma}
\begin{proof}
If $u_1\neq 0$, then columns one, three, four and six are linearly
independent. So assume henceforth $u_1=0$, that is,
\[
A=
\left(\begin{matrix}
0 & 0 & 0 & 0 & 0 & 0 & 0\\
0 & 0 & 0 & 0 & 0 & 0 & 0\\
{u_2} & {u_3} & 0 & 0 & 0 & 0 & 0\\
{u_4} & 4{u_2} & 0 & 0 & 0 & 0 & 0\\
{u_5} & -2{u_4} & 0 & 0 & 0 & 0 & 0\\
{u_6} & 0 & 2{u_4} & -4{u_2} & -{u_3} & 0 & 0\\
0 & -{u_6} & -{u_5} & -{u_4} & -{u_2} & 0 & 0
\end{matrix}\right).
\]

First, assume $u_2=0$. Then columns one, two, three, four are
linearly independent, unless also $u_4=0$. This is case (a).

Next, assume $u_2\neq0$. If $u_4=0$, we see that columns one, two and
four are linearly independent. So assume also $u_4\neq0$.
Given $\rk A=2$, the subcolumns
$\left(\begin{smallmatrix} u_2\\u_4\\u_5\end{smallmatrix}\right)$
and
$\left(\begin{smallmatrix} u_3\\4u_2\\-2u_4\end{smallmatrix}\right)$
must be linearly dependent. So add $-\frac{u_3}{u_2}$ times the first
column of $A$ to its second column, and obtain the identities
$4 u_2 - \frac{u_3 u_4}{u_2}=0$ and $-2 u_4 - \frac{u_3 u_5}{u_2}=0$.
By solving for $u_3$ and $u_5$ we obtain
\[
u_3 = \frac{4 u_2^2}{u_4},
\quad
u_5 = -\frac{u_4^2}{2u_2}.
\]
Plugging this into $A$ we can verify directly that $\rk A=2$.
This is case (b).
\end{proof}

\begin{remark*}
Note that condition (a) in Lemma \ref{lem:rank2} is necessary but not
sufficient for $\rk A=2$.
Also, the condition $\rk A\leq 2$ already implies $\rk A=2$ for the
non-zero matrices in case (a) or (b).
\end{remark*}

Below, we use the notation $A(u_3,u_5,u_6)$ for matrices of case (a)
in Lemma \ref{lem:rank2}, and $B(u_2,u_4,u_6)$ for matrices of case
(b),
\begin{align*}
A(u_3,u_5,u_6)
&=\left(\begin{matrix}
0 & 0 & 0 & 0 & 0 & 0 & 0\\
0 & 0 & 0 & 0 & 0 & 0 & 0\\
0 & {u_3} & 0 & 0 & 0 & 0 & 0\\
0 & 0 & 0 & 0 & 0 & 0 & 0\\
{u_5} & 0 & 0 & 0 & 0 & 0 & 0\\
{u_6} & 0 & 0 & 0 & -{u_3} & 0 & 0\\
0 & -{u_6} & -{u_5} & 0 & 0 & 0 & 0
\end{matrix}\right), \\
B(u_2,u_4,u_6)
&=\left(\begin{matrix}
0 & 0 & 0 & 0 & 0 & 0 & 0\\
0 & 0 & 0 & 0 & 0 & 0 & 0\\
u_2 & \frac{4 u_2^2}{u_4} & 0 & 0 & 0 & 0 & 0\\
u_4 & 4 u_2 & 0 & 0 & 0 & 0 & 0\\
 -\frac{u_4^2}{2u_2} & -2u_4 & 0 & 0 & 0 & 0 & 0\\
{u_6} & 0 & 2u_4 & -4u_2 & -\frac{4 u_2^2}{u_4} & 0 & 0\\
0 & -{u_6} & \frac{u_4^2}{2u_2} & -u_4 & -u_2 & 0 & 0
\end{matrix}\right),\quad u_2,u_4\neq 0.
\end{align*}

\begin{lemma}\label{lem:no_rank2}
The maximal nilpotent subalgebra $\frm$ of $\frg_{2(2)}$
does not contain a three-dimensional Lie sub\-algebra
$\frb$ consisting only of matrices of rank two (and zero).
\end{lemma}
\begin{proof}
Rank two matrices in $\frm$ belong to one of the two types described
in Lemma \ref{lem:rank2}.
First, assume $\frb$ contains only matrices of the form
$A(u_3,u_5,u_6)$.
In order to be of rank two, one of the parameters $u_3,u_5,u_6$
must be zero in each element of $\frb$.
But the set of matrices with this property is a finite union
of two-dimensional subspaces, contradicting $\dim\frb=3$.

So assume $\frb$ contains a non-zero matrix $B(u_2,u_4,u_6)$.
If $\frb$ contains two elements $B_1=B(u_2,u_4,u_6)$, $B_2=B(v_2,v_4,v_6)$
that are not multiples of each other, then one of
$u_2+v_2\neq 0$ or $u_4+v_4\neq 0$ holds.
The sum $B_{12}=B_1+B_2$ is
\[
B_{12}=
\begin{pmatrix}
0 & 0 & 0 & 0 & 0 & 0 & 0\\
0 & 0 & 0 & 0 & 0 & 0 & 0\\
{u_2}+{v_2} & \frac{4u_2^2}{u_4}+\frac{4v_2^2}{v_4} & 0 & 0 & 0 & 0 & 0\\
u_4+v_4 & 4 u_2+4 v_2 & 0 & 0 & 0 & 0 & 0\\
-\frac{u_4^2}{2u_2}-\frac{v_4^2}{2v_2} & -2u_4-2v_4 & 0 & 0 & 0 & 0 & 0\\
u_6+v_6 & 0 & 2u_4+2v_4 & -4u_2-4v_2 & -\frac{4u_2^2}{u_4}-\frac{4v_2^2}{v_4} & 0 & 0\\
0 & -u_6-v_6 & \frac{u_4^2}{2u_2}+\frac{v_4^2}{2v_2} & -u_4-v_4 & -u_2-v_2 & 0 & 0
\end{pmatrix}.
\]
If $\rk B_{12}=2$, then, since one of $u_2+v_2$ and $u_4+v_4$ is non-zero,
both must be non-zero ($B_{12}$ is case (b) of Lemma
\ref{lem:rank2}). Set $w_2=u_2+v_2$, $w_4=u_2+v_4$.
By Lemma \ref{lem:rank2}, entry $(3,2)$ of $B_{12}$ must be
$w_3 = \frac{4 w_2^2}{w_4}=4\frac{u_2^2+ 2u_2u_4 + u_4^2}{u_4+v_4}$.
But this entry is $\frac{4 u_2^2}{u_4}+\frac{4 v_2^2}{v_4}$, and the
resulting equation
$4\frac{u_2^2+ 2u_2u_4 + u_4^2}{u_4+v_4}
=
\frac{4 u_2^2}{u_4}+\frac{4 v_2^2}{v_4}$
simplifies to $u_2 v_4 - u_4 v_2 = 0$.
Solving for $v_2$ and plugging back into
$B(v_2,v_4,v_6)$ shows that
\[
B_1-\frac{u_4}{v_4}B_2=A\Bigl(0,0,u_6-\frac{u_4}{v_4}v_6\Bigr).
\]
If this is zero, then $B_1$ is a multiple of $B_2$, in contra\-diction
to our assumption.

So $\frb$ contains $B_1$ and $A(0,0,1)$ as basis
elements, and all other elements $B(v_2,v_4,v_6)\in\frb$ lie in
the two-dimensional subspace $U$ spanned by those two elements.
Since $\dim\frb=3$, there exists a third basis vector of $\frb$
of the form $B(u_2,u_4,u_6)+A(w_3,w_5,0)$ for some
$u_2,u_4\neq0$, and at least one of $w_3$ or $w_5$ is non-zero,
for otherwise the element lies in the subspace $U$.
But then the resulting matrix cannot be of rank two, as it is neither
of type (a) nor type (b) in Lemma \ref{lem:rank2}. In fact,
it is not of type (a) since $u_2\neq0$, and for type (b), the entries
$u_3$, $u_5$ are functions of $u_2$, $u_4$, which prohibits
adding a non-zero term $w_3$ or $w_5$.

In conclusion, a Lie subalgebra with the properties required for
$\frb$ does not exist in $\frm$.
\end{proof}

\begin{lemma}\label{lem:no_adn_embedding}
Let $\frn$ be a seven-dimensional two-step nilpotent Lie
algebra with an invariant scalar product of index three.
Then $\ad(\frn)$ is not contained in $\frg_{2(2)}$.
\end{lemma}
\begin{proof}
The invariance of $\met$ and two-step nilpotence of $\frn$
imply $\zen(\frn)=\frw\oplus\frj$ and
$[\frn,\frn]=[\frj^*,\frj^*]=\frj$.
In particular, $\dim\ad(\frn)=\dim\frj^*=\dim\frj$.
By assumption, $\frj\neq\zsp$, hence $\dim\frj=3$.
In fact, if $\dim\frj^*\leq2$, then $[\frj^*,\frj^*]=\zsp$ due to the
invariance of $\met$, a contradiction.
Also by invariance, $\im\ad(x)=x^\perp\cap\frj$ for all non-zero
$x\in\frj^*$.
So $\rk\ad(x)=2$ for all non-zero $x\in\frj^*$.

If $\ad(\frn)$ is contained in $\frg_{2(2)}$, then it is a
nilpotent triangular sub\-algebra, and as such contained in a maximal
nilpotent triangular sub\-algebra.
This means $\ad(\frn)$ is conjugate to a three-dimensional
abelian subalgebra of $\frm$ in \eqref{eq:triangular} consisting of
rank two matrices (and zero).
But by Lemma \ref{lem:no_rank2}, such a subalgebra cannot exist
in $\frm$.
\end{proof}

%

\begin{lemma}\label{lem:no_embedding}
Let $\frn$ be a seven-dimensional nilpotent but non-abelian Lie algebra
with an invariant scalar product of index three.
Then $\ad(\frn)$ is not contained in $\frg_{2(2)}$.
\end{lemma}
\begin{proof}
If $\ad(\frn)$ is contained in $\frg_{2(2)}$, then it is a
nilpotent triangular sub\-algebra, and as such contained in a maximal
nilpotent triangular sub\-algebra. This means $\ad(\frn)$ is conjugate
to a subalgebra of the maximal nilpotent subalgebra $\frm$ of
$\frg_{2(2)}$ given by \eqref{eq:triangular}.
The possibilities for $\frn$ are covered by the examples in Section \ref{sec:liealgebra}, up to some abelian factor.
We use the notation from these examples.

Suppose first that $\frn$ is indecomposable. Then $\frn$ must be the
Lie algebra $\frnI$ from Example \ref{ex:kath6a}.
Now $\ad_{\frnI}(\frw)$ is a three-dimensional abelian subalgebra
of $\ad_{\frnI}(\frnI)$, generated by the elements $\ad(w_i)$,
$i=1,2,3$.
From the commutator relations of $\frnI$ it follows that the image of
every non-zero $Q\in\ad_{\frnI}(\frw)$ is the span of $z_1$ and
$Q a_1$. Hence $\rk Q=2$ for all non-zero $Q$. Now Lemma
\ref{lem:no_rank2} tells us that such a subalgebra does not exist in
$\frm$ (alternatively, this case could be excluded by
comparison with Kath's classification result \cite[Theorem 6.8]{kath2}).

Suppose now that $\frn$ is decomposable, that is, an orthogonal direct
sum $\frg=\frn_1\times\frn_0$ of non-zero metric Lie algebras. We may
assume $\dim\frn_1>\dim\frn_0$. Then $1\leq\dim\frn_0\leq 4$, so that
$\frn_0$ is abelian by Lemma \ref{lem:dim4}.
If $\frn_1$ is not abelian itself, then $5\leq\dim\frn_1\leq 6$.
This leaves $\frnII$ and $\frnIII$ from Examples \ref{ex:kath7a} and
\ref{ex:kath8} as possibilities for $\frn_1$.
Since $\frnII\times\RR$ is two-step nilpotent, its adjoint
representation cannot be contained in $\frg_{2(2)}$ by Lemma
\ref{lem:no_adn_embedding}.

So assume $\frn=\frnIII\times\RR^2_1$.
By the relations in Example \ref{ex:kath8}, the algebra
$\ad(\frn)$ is three-dimensional, spanned by $\ad(a_1)$,
$\ad(a_2)$ and $\ad(w)$.
Furthermore, it follows that the image of a linear combination
$Q=\alpha_1\ad(a_1)+\alpha_2\ad(a_2)+\beta\ad(w)$ equals
$\Span\{\alpha_1 w - \beta\epsilon z_1, \alpha_2\epsilon z_1-\alpha_1\epsilon z_2\}$.
So $\rk Q=2$ for all non-zero $Q\in\frb$.
We can now apply Lemma
\ref{lem:no_rank2} to conclude that $\ad(\frn)$ is not
contained in $\frg_{2(2)}$.
\end{proof}

%
%
%
%



\section{Proofs of the main theorems}
\label{sec:proofs}

\begin{proof}[Theorem \ref{mthm:biinvariant}]
In this situation, Lemma \ref{lem:dimG7} applies.
This means $G$ is nilpotent, $\dim G=7$ and $\ad(\frg)$ is a subalgebra
of $\frg_{2(2)}$. But by Lemma \ref{lem:no_embedding}, this is not possible unless $\frg$ and hence $G$ is abelian.
Since $G$ acts effectively on $M=G/\,\Gamma$, the lattice $\Gamma$ must
be trivial, which implies that $G$ itself is compact.
This means $G=M$ is a torus $\RR^7_3/\ZZ^7$.
\end{proof}

\begin{proof}[Theorem \ref{mthm:flat}]
By Lemma \ref{lem:Levi_noncompact}, the group $G$ has dimension seven
and $M=G/\,\Gamma$ for some uniform lattice $\Gamma$ in $G$, and the
metric $\g_M$ on $M$ is induced by a bi-invariant metric $\g_G$ on $G$.
Now we are in the situation of Theorem \ref{mthm:biinvariant}, which
concludes the proof.
\end{proof}






\end{document}